\documentclass[11pt,reqno]{amsart} 

\usepackage[utf8]{inputenc} 
\usepackage{xcolor}

\usepackage{geometry} 
\usepackage{color,graphicx} 

\usepackage{amsmath}
\usepackage{amssymb}
\usepackage{amsthm}
\usepackage{chemarr}
\usepackage{esint}
\usepackage{cite}
\usepackage{booktabs} 
\usepackage{array} 
\usepackage{paralist} 
\usepackage{mathrsfs}
\usepackage{verbatim} 
\usepackage{dsfont}
\usepackage{mathtools}
\usepackage[yyyymmdd,hhmmss]{datetime}
\usepackage{subcaption}
\usepackage{hyperref}
\usepackage{fancyhdr} 
\newtheorem{thm}{Theorem}[section]
\newtheorem{defi}[thm]{Definition}

\newtheorem{lma}[thm]{Lemma}

\theoremstyle{definition}
\newtheorem{rem}[thm]{Remark}

\DeclareMathOperator{\diam}{diam}

\DeclareMathOperator{\size}{size}

\DeclareMathOperator{\dist}{dist}

\DeclareMathOperator{\conv}{conv}

\DeclareMathOperator{\Lin}{Lin}

\newcommand{\id}{\mathrm{id}}

\newcommand{\e}{\eps}

\newcommand{\wto}{\rightharpoonup}
\renewcommand{\H}{{\mathscr{H}}}
\renewcommand{\d}{\mathrm{d}}

\newcommand{\Gr}{\mathrm{Gr}\,}

\newcommand{\CR}{\mathrm{CR}}
\newcommand{\nn}{\mathbf{n}}
\newcommand{\DD}{\nabla}
\newcommand{\snu}{\mathscr{S}_N(U)}

\usepackage{accents}
\newcommand{\ubar}[1]{\underaccent{\bar}{#1}}

\def\R{{\mathbb R}} 

\def\E{{\mathscr E}}
\def\T{{\mathscr T}}

\def\V{{\mathscr V}}

\newcommand{\eps}{\varepsilon}

\def\XXint#1#2#3{{\setbox0=\hbox{$#1{#2#3}{\int}$ }
\vcenter{\hbox{$#2#3$ }}\kern-.6\wd0}}

\begin{document}

\title{Consistent and convergent discretizations of Helfrich-type energies on general meshes}
\date{\today} 
\author[V.~Degrooff]{Vincent Degrooff}
\author[P.~Gladbach]{Peter Gladbach}
\author[H.~Olbermann]{Heiner Olbermann} 
\address[Peter Gladbach]{Institut f\"ur Angewandte Mathematik, Universit\"at Bonn, 53115 Bonn, Germany}
\address[Heiner Olbermann]{Institut de Recherche en Math\'ematique et Physique, UCLouvain, 1348 Louvain-la-Neuve, Belgium}
\address[Vincent Degrooff]{Institute of Mechanics, Materials and Civil Engineering,  UCLouvain, 1348 Louvain-la-Neuve, Belgium}

\email[Peter Gladbach]{gladbach@iam.uni-bonn.de}
\email[Heiner Olbermann]{heiner.olbermann@uclouvain.be}
\email[Vincent Degrooff]{vincent.degrooff@uclouvain.be}

\begin{abstract}
We analyze discrete  approximations of the second fundamental form on graphs of functions that are piecewise affine on irregular meshes. Being related with the Morley finite element, the approximation in this precise form has first been suggested by E.~Grinspun, Y.~Gingold, J.~Reisman, and D.~Zorin, \emph{Computer Graphics Forum} 2006, Volume 25, 547-556. 
  We show how to use this framework to  approximate  continuous variational problems of the form $E_0(M) := \int_M f(x,n_M(x),D n_M(x))\,d\H^2(x)$, where $n_M$ denotes the normal of the surface $M$. Here the integrand $f$ is not necessarily quadratic. This corresponds to nonlinear Euler-Lagrange equations. Our approximation is rigorously formulated in the framework of $\Gamma$-convergence: We combine an ansatz-free asymptotic lower bound for any uniform approximation  and a recovery sequence consisting of any regular triangulation of the limit sequence and an almost optimal choice of edge director. We give numerical examples showing the efficiency and accuracy of the algorithm in nonlinear problems.
\end{abstract}


\maketitle

\section{Introduction}
Discrete notions of curvature for embedded surfaces are important to various applications: simulation of elastic plates and shells, surface regularization, and partial differential equations on surfaces.

In the smooth setting, curvature energies of surfaces are well-studied. These energies are typically of integral form
\begin{equation}\label{eq: integral energy}
  E_0(M) := \int_M f(x,n_M(x),D n_M(x))\,d\H^2(x)\,.
\end{equation}

Here $M\subset \R^3$ is an immersed smooth surface, $n_M:M\to S^2$ its oriented unit normal field, and $D n_M(x)\in \mathrm{Lin}(T_xM;T_xM)$ the \emph{shape operator} of $M$ at $x$, which encodes the extrinsic curvature of $M$.

The energies of the type \eqref{eq: integral energy} are called \emph{Helfrich-type energies}, with particular attention given to the \emph{Willmore energy}
\begin{equation}\label{eq: Willmore energy}
W(M) := \int_M |D n_M(x)|^2\,d\H^2,
\end{equation}
which has links to plate theory, mean curvature flow, and conformal geometry.

A number of discrete versions of energies \eqref{eq: integral energy} and \eqref{eq: Willmore energy} have been proposed, see e.g.~\cite{bobenko2008discrete,meyer2003discrete,grinspun2006discrete,crane2017glimpse,schumacher2019variational,schumacher2020variational}.
Typically, the objects studied are not smooth surfaces but \emph{triangular complexes}, which are finite collections of triangles $\T := \{\kappa_1,\ldots,\kappa_N\}$ in $\R^3$ that are glued together pairwise along edges. These objects appear in a variety of contexts, including computer graphics, finite elements, and algebraic topology.

A widely used discrete version of the Willmore energy \eqref{eq: Willmore energy} is the finite-difference model introduced in \cite{grinspun2003discrete}
\begin{equation}\label{eq: differences}
E_{FD}(\T)=\frac12 \sum_{\kappa,\kappa'\,n.n.} \frac{l_{\kappa,\kappa'}}{d_{\kappa,\kappa'}}|\bar n(\kappa) - \bar n(\kappa')|^2,
\end{equation}
for which the second and third author recently \cite{gladbach2021approximation} showed $\Gamma$-convergence to the Willmore energy \eqref{eq: Willmore energy}. Here, $\bar n(\kappa)\in S^2$ denotes the oriented unit normal of a triangle $\kappa\in \T$, $l_{\kappa,\kappa'}>0$ the edge length between nearest neighbors $\kappa,\kappa'\in \T$, and $d_{\kappa,\kappa'}>0$ the distance between the circumcenters of $\kappa,\kappa'$.

However, as already seen in \cite{grinspun2003discrete} and \cite{schmidt2012universal}, the construction of an approximating sequence $\T_h\to M$ with $E_{FD}(\T_h)\to W(M)$ is extremely sensitive to the chosen mesh. Low-energy configurations must be defined on meshes with  nearly right angles.

In this article, we instead study a different discrete energy introduced in \cite{grinspun2006computing},
\begin{equation}
\label{eq:2}
  E(\T,n) := \sum_{\kappa\in \T} \int_{\kappa} f(x, \bar n(\kappa), D n_\kappa)\,d\H^2(x).
\end{equation}

Here $\bar n(\kappa)\in S^2$ is the oriented unit normal of $\kappa$ and $n:\E(\T)\to \R^3$ is a so-called \emph{edge director}, which can be chosen subject to some constraints on every edge $e\in\E(\T)$ and is then linearly interpolated on every triangle $\kappa\in \T$ between the three edge midpoints, yielding the piecewise constant gradient $D n_\kappa\in \Lin(T\kappa;\R^3)$. We note here that the piecewise affine interpolation is not continuous along the edges, only on the edge midpoints, and thus belongs to a geometric version of the Crouzeix-Raviart nonconforming finite element space \cite{crouzeix1973conforming}, and is similar to the Morley finite element \cite{morley1968triangular,ciarlet2002finite}.

\medskip

We will show under the assumption that all surfaces are graphs, that the cluster points of the minimizers of the discrete variational problem \eqref{eq:2} are minimizers of the continuous problem \eqref{eq: integral energy}. Furthermore we show that every (graphical) surface of regularity $W^{2,p}$ can be approximated by discrete ones simultaneously in $W^{1,p}$ and in energy. In the parlance of  \cite{grinspun2006computing}, these results amount to the statement that \eqref{eq:2} is a  ``convergent'' and ``consistent'' discretization of \eqref{eq: integral energy}. The convergence and consistency for problems with linear Euler-Lagrange equation (which corresponds to the case of a quadratic function $\xi \mapsto f(x,n,\xi)$) have already been addressed in \cite{grinspun2006computing,zorin2005curvature}. Our  treatment allows to establish a more general statement encompassing also nonlinear problems.

\medskip

To frame our result in yet another manner, Theorem \ref{thm: main} below is the statement that under the assumption that all surfaces are graphs over the same domain, the discrete energy $E(\T,n)$ $\Gamma$-converges to its continuous counterpart $E_0(M)$. The above statement about cluster points follows from the latter by   general facts from  the theory of $\Gamma$-convergence \cite{MR1201152,MR1968440}.

\medskip

The proof of the compactness and lower bound part of Theorem \ref{thm: main} relies on recasting the discrete surface normal as an element in the so-called Crouzeix-Raviart space \cite{crouzeix1973conforming}, for which suitable compactness results exist \cite{buffa2008broken,2011-IMAJNA-lavrentiev_nc}. In order to identify the limit of the edge directors with the limit of the discrete surface normal, one needs appropriate estimates, which we obtain in Lemma \ref{lma: normal estimate}.  The upper bound part of Theorem \ref{thm: main} follows from a judicious construction of the edge director (the idea of which can already be found in \cite{grinspun2006computing}) and some calculus.

\medskip

The plan of the paper is as follows: After introducing  some notation, we state the main result  in  Section \ref{sec: discrete}. The proof is given  in Section \ref{sec: compactness}. In Section \ref{sec:numerical} we give some numerical examples.  

\section{Setup and statement of  result}\label{sec: discrete}

\subsection{General notation}
The symbol ``$C$'' will be used as follows: A statement such as ``$a\leq Cb$'' has to be understood as ``there exists a constant $C>0$ such that $a\leq Cb$''. We also write $a\lesssim b$ in that situation. The $l$-dimensional Hausdorff measure will be denoted by $\H^l$. For a sequence $f_h$ in a Banach space $W$ converging weakly to $f\in W$, we write $f_h\wto f$ in $W$. 

\medskip

For points $p_1,p_2\in\R^3$, we denote the line segment connecting them by
\[
  [p_1,p_2]:=\{tp_1+(1-t)p_2:t\in[0,1]\}\,.
  \]

\subsection{Triangular complexes}

We define triangular complexes as follows:
\begin{defi}
  \begin{itemize}
  \item A \emph{triangle} is a set $\kappa = \conv(x,y,z)\subset \R^3$
    with $x,y,z\in \R^3$ not colinear. The \emph{vertices} of $\kappa$ are the points $x,y,z$, and the edges of $\kappa$ are the line segments $[x,y],[y,z], [x,z]\subset \R^3$.
\item The \emph{unit normal} of the
    triangle $\conv(x,y,z)$ is one of the vectors
  \[
  \bar n(\kappa) := \pm \frac{(y-x)\times(z-x)}{\left|(y-x)\times (z-x)\right|}.
  \]
\item  A \emph{triangular complex} is a finite family of triangles $\T := \{\kappa_1,\ldots,\kappa_N\}$ with the property that the intersection of two different triangles $\kappa, \kappa'\in \T$ is either empty, a single common vertex, or an entire common edge. 
\item The set of \emph{vertices}  of the triangular complex $\T$  is denoted by $\V(\T)$, and contains all three vertices of all triangles $\kappa\in \T$. The set of \emph{edges}  of the triangular complex $\T$  is denoted by $\E(\T)$, and contains all three edges of all triangles $\kappa\in \T$.
\item The size of $\T$ is given by the maximal diameter of any $\kappa\in \T$, 
\[
\size(\T)=\max_{\kappa\in\T}\diam(\kappa)\,.
\]
\end{itemize}
\end{defi}

\begin{defi}\label{def:regular}
Let $C_*>1$. We will say that $\T$ is regular if for every $\kappa\in \T$, 
  \begin{equation}\label{eq:9}
  \H^2(\kappa)\geq C_* (\diam(\kappa))^2.
\end{equation}
\end{defi}

From now on it will always be understood that the triangular complexes we are considering are regular. Also, the generic constant ``$C$'' (which is implicit in statements such as $a\lesssim b$) may depend on $C_*$ in the sequel. Like $C$, the constant $C_*$ may change its value in the course of the proof. However  when considering sequences of regular triangular complexes $(\T_h)_{h>0}$, it will be independent of the sequence parameter $h$.  

\bigskip

For a triangular complex $\T$, let $m(\T)$ denote the set of midpoints of edges in $\T$. 
  Given a regular triangular complex $\T$, we define the notion of edge director field $n:\E(\T)\to \R^3$ below. In this definition, we write $\tau(e)=\frac{y-x}{|y-x|}$ for an edge $e=[x,y]\in \E(\T)$. (The sign ambiguity in the definition of $\tau$ does not have any consequence in its usage.)

\begin{defi}
  A \emph{unit edge director} is a map $n:\E(\T)\to S^2$ such that $n(e)\cdot \tau(e) = 0$ and $n(e) \cdot \bar n(\kappa) \geq 0$ for all $e = \kappa\cap \kappa' \in  \E(\T)$. The family of unit edge directors is denoted $N(\T)$.
  
\end{defi}

To each edge director field we associate a piecewise affine  interpolation:
\begin{defi}
Let $n\in N(\T)$. Define the piecewise-affine (but discontinuous) interpolation, also denoted $n:\bigcup_{\kappa\in \T}\kappa \to \R^3$, by defining its restriction to  $\kappa\in \T$ as the unique affine map $n|_{\kappa}:\kappa\to \R^3$ that coincides with $n$ on the edge midpoints $\frac{x+y}{2}$. Its \emph{piecewise gradient} $D n_\kappa\in \Lin(T\kappa;\R^3)$ is extended to a $3\times 3$ matrix by precomposition with the orthogonal projection to $T\kappa$, so that we may write $D n_\kappa\in \R^{3\times 3}$. 
\end{defi}

\subsection{Definition of the discrete energy}

We now present a discretization of the Helfrich-type energy \eqref{eq: integral energy} where the integrand
\[
  f:\R^3\times S^2 \times \R^{3\times 3}\to\R
\]
satisfies the assumptions
\begin{itemize}
  \item [(A1)] $f:\R^3\times S^2 \times \R^{3\times 3}\to \R$ is continuous in all its variables,  and convex in the last.
  \item [(A2)] $f(x,n,A)\gtrsim |A|^p$ for some $p\in(1,\infty)$.
\end{itemize}

 We will consider discrete energies $E:\{(\T,N(\T))\} \to \R$, namely
\[
E (\T,n) := \sum_{\kappa \in \T} \int_\kappa f(x, \bar n(\kappa), D n_\kappa)\,d\H^2(x).
\]

\subsection{Piecewise affine graphs over triangulations}
\label{sec:notation}

We  introduce the notation that will help us exploit the assumption that all surfaces are graphs. 

\medskip

Assume that $U\subset\R^2$ is an open polygon.
Let $\T$ be a regular triangulation of $U$, by which we mean a triangular complex consisting of triangles that are immersed in $\R^2$ with 
\[
\bigcup_{\kappa\in\T}\kappa=\overline U\,.
\]
(For the rest of the present section  the symbol $\T$ will  denote  triangular complexes of this kind, with the exception of Remark \ref{rem:gamma}.) Now let $u$ be a  continuous function $U\to \R$ whose restriction to each $\kappa\in \T$ is affine. Then we define the \emph{push-forward} of the triangulation $\T$ under $u$ as the triangular complex consisting of triangles that are immersed in $\R^3$, 
\[
u_*(\T):=\left\{ \{(x,u(x)\}:x\in \kappa,\kappa \in\T\right\}\,.
\]

\medskip

We consider 3-tuples 
\[
(\T,u,n)
\]
where $\T$ is a regular triangulation of $U$, $u$ is a piecewise affine continuous function $U\to \R$ which is affine on each $\kappa\in \T$, and $n\in N(u_*(\T))$. Let the set of such 3-tuples be denoted by
\[
  \snu\,.
\]

\medskip

Let $P_1(\T)$ denote the set of piecewise affine functions relative to $\T$,
\[
  P_1(\T)=\left\{v\in L^1\left(\bigcup_{\kappa\in\T}\kappa\right):v|_\kappa \text{ is affine }\forall \kappa\in \T\right\}\,.
\]
The (first order) Crouzeix-Raviart finite element space is defined as
\[
  \CR(\T)=\left\{v\in P_1(\T):v\text{ is continuous in }m(\T)\right\}\,.
  \]
Any edge  director field $n\in N(u_*(\T))$ can  be associated to an element of $\CR(\T;\R^3)$ as follows. For an edge midpoint $p\in m(\T)$, set $\ubar{n}(p)=n(p,u(p))$. Then define $\ubar{n}$ by affine interpolation on every $\kappa\in \T$. This defines $\ubar n$ as an element of $\CR(\T;\R^3)$. In order to alleviate the notation, we will not distinguish $\ubar n$ and $n$ in the sequel. No confusion will arise from this.
The piecewise gradient of this map will be denoted by 
\[
\DD n:U\to \R^{3\times 2}\,.
\]
This implies that the discrete shape operator $Dn|_{\kappa}$ and the piecewise gradient $\nabla n$ are related by 
\[
Dn|_{\kappa}(\id_{2\times 2}+e_3\otimes\nabla u)= \nabla n\,,
\]
where $\id_{2\times 2}$ denotes the two-by-two identity matrix. 

\medskip

In order to express the surface normal to the graph of $u$ in terms of that function, we define $\nn:\R^2\to S^2$,
\[
\nn(p):=\frac{(-p^\bot,1)}{\sqrt{|p|^2+1}}\,,
\]
where $p^\bot=(-p_2,p_1)$ for $p=(p_1,p_2)$.
Now we set
  \begin{equation}\label{eq:4}
\begin{split}
F:U\times \R\times \R^2\times \R^{3\times 2}&\to \R\\
(x,z,p,\xi)&\mapsto f((x,z),\nn(p),\xi(\id_{2\times 2} + e_3\otimes p)^{-1})\sqrt{1+|p|^2}\,,
\end{split}
\end{equation}
which allows us to write 
  \begin{equation}\label{eq:5}
E(u_*(\T), n)=\int_U F(x,u(x),\nabla u(x),\DD n(x))\d x\,.
\end{equation}

Analogously, if $M =\Gr u$ for some $u\in W^{2,p}\cap W^{1,\infty}(U)$, then
\[
E_0(M) = \int_U F(x,u(x), \nabla u(x), \nabla (\nn(\nabla u))(x))\d x.  
\]

\subsection{Statement of the result}

As above, we assume that $U\subset \R^2$ is a polygonal domain.
\begin{thm}
  \label{thm: main}
  \begin{itemize}
\item[(i)]
Let $(\T_h,u_h,n_h)$ be a sequence in $\snu$  with $\size(\T_h)\leq h$, and the uniform bound $\|u_h\|_{W^{1,\infty}}\leq L$. Furthermore assume that
\[
\sup_{h>0} \int_U F(x,u_h,\nabla u_h, \DD n_h)\d x\leq  C\,.
\]

Then there exists a subsequence (no relabeling) and some $u\in W^{2,p}(U)$ with surface normal $\nn(\nabla u)$ such that $u_h\to u$ uniformly and 
\[
\begin{split}
  n_h&\to \nn(\nabla u) \quad \text{ in }L^q(U)  \text{ for all } q\in [1,\infty)\\
  n_h & \to \nn(\nabla u) \quad \text{ in  }L^\infty(U) \text{ if }p>2\\
\nabla n_h&\wto \nabla \left(\nn(\nabla u)\right) \quad \text{ in }L^p(U)
 \,.
\end{split}
\]

\item[(ii)] 
  Let $(\T_h,u_h,n_h)$ and $u$ be as in (i). Then
  \begin{equation}\label{eq: lower semicontinuity}
    \liminf_{h\to 0}\int_U F(x,u_h,\nabla u_h,\DD n_h)\d x\geq \int_U F(x,u,\nabla u,\nabla (\nn(\nabla u)))\d x.
  \end{equation}
  \item[(iii)] 
  Let $u\in W^{1,\infty}\cap W^{2,p}( U)$,  and $\T_h$ be a sequence of regular triangulations of $U$ with $\size\T_h=h$. Then there exist $u_h$ and $n_h$ such that $(\T_h,u_h,n_h)\in \snu$  and
  \[
    \lim_{h\to 0} \int_U F(x,u_h,\nabla u_h,\DD n_h)\d x=\int_U F(x,u,\nabla u,\nabla (\nn(\nabla u)))\d x\,.
    \]
  If $u\in C^2(\bar U)$ we may choose $u_h$ the piecewise affine interpolation of $u$ itself.
\end{itemize}
\end{thm}

\begin{rem}
\label{rem:gamma}
  \begin{itemize}
  \item[(i)] With the integrand $F$ chosen as in \eqref{eq:4},  the above result shows the discrete-to-continuum convergence
\[
E(\T,n)\stackrel{\Gamma}{\to}E_0
\] 
in the sense of $\Gamma$-convergence,
under the assumption that all surfaces are Lipschitz graphs over the same domain $U$.
\item[(ii)] In light of the compactness result \cite{langer1984compactness} for immersions of a fixed compact closed surface $\Sigma$ for $p>2$, we may replace the Lipschitz graph condition with the topological condition that every $\T$ be a bounded immersion of a fixed compact closed surface, and obtain the analogue of Theorem \ref{thm: main}. The reason, as shown in \cite{langer1984compactness}, is that the Sobolev embedding $W^{1,p}\to C^{0,\alpha}$ ensures that surfaces with finite energy are locally  Lipschitz graphs.
\item[(iii)] We may state and prove an analogous result for discrete approximations via the definition of \emph{pseudo-unit edge directors} instead of  unit edge directors. This is more efficient when it comes to computations. For further comment on this, see Section \ref{sec:numerical}.
\item[(iv)] Note that by the compactness part of the above theorem, the convergence of energies implies the strong $L^p(U)$ convergence of the discrete shape operator $Dn_h$ to $Dn_M$, as long as $f$ is strictly convex.
  \end{itemize}
\end{rem}

\section{Proof of Theorem \ref{thm: main}}

\label{sec: compactness}

The following lemma states that any unit edge director is close to the actual triangle normal field. Recall that $\bar n$ denotes the (piecewise constant) surface normal of the triangular complex $\T$.

\begin{lma}\label{lma: normal estimate}
  Let $\T$ be regular. Then
for every $n\in N(\T)$, $e = \kappa \cap \kappa'\in \E(\T)$, 
\[|n(e) - \bar n(\kappa)| \lesssim (\diam\kappa) |D n_\kappa|.\]
\end{lma}

\begin{proof}
  For $e=[x,y]$, we write $\tau(e)=\frac{y-x}{|y-x|}$. (Again, the sign ambiguity in the definition of $\tau$ will not have any consequence in the sequel.) Let $e'$ be another edge bordering $\kappa$.
By rescaling and a rigid motion we can assume $\diam\kappa \simeq 1$, and that $\tau(e)= (0,1,0)$, $\tau(e') = (s,t,0)$ with
  \[
    \begin{split}
    C^{-1}\leq |s|&\leq C\\
    |t|&\leq C\,.
  \end{split}
\]
This also implies $\bar n(\kappa)=(0,0,1)$. 
Since $n(e)\cdot \tau(e) = 0$, we may write $n(e) = (\alpha,0,\beta)$ and $n(e') = (-t/s\gamma,\gamma,\delta)$ for some $\alpha,\beta,\gamma,\delta\in \R$. Thus
\[
  \begin{split}
    |D n_\kappa| &\gtrsim |n(e) - n(e')| \\
    &\gtrsim \min_{r:|r|\lesssim 1}\left(|\alpha-r\gamma|+|\gamma|\right)\\
    &\gtrsim  |\alpha| + |\gamma|.
  \end{split}
\]
By $|\alpha|=|\bar n(\kappa)\times n(e)|$ we obtain in particular
\begin{equation}\label{eq: cross product}
 |\bar n(\kappa)\times n(e)|\lesssim |D n_\kappa|.  
\end{equation}

To conclude,  we use the assumption $n(e) \cdot \bar n(\kappa) \geq 0$, and  obtain 
\[
  |n(e)-\bar n(\kappa)| \lesssim |n(e)\times \bar n(\kappa)| \lesssim |D n_\kappa|.
\]
\end{proof}


\begin{proof}[Proof of Theorem \ref{thm: main} (i)]

The uniform convergence of a subsequence (no relabeling) $u_h\to u$ to an $L$-Lipschitz continuous function $u\in W^{1,\infty}(U)$ is immediate from the Arzel\`a-Ascoli theorem. 


\medskip

By the growth assumptions on $F$ in its last variable, we have that
  \begin{equation}\label{eq:6}
\int_U |\nabla n_h|^p\d x\leq C\,.
\end{equation}
Hence by Theorem 4.3 of \cite{2011-IMAJNA-lavrentiev_nc}\footnote{Note that \cite{2011-IMAJNA-lavrentiev_nc} assumes a stronger, global mesh regularity than Definition \ref{def:regular}. However, the proof of Theorem 4.3 of \cite{2011-IMAJNA-lavrentiev_nc} can be used verbatim in our setting.}, we have that there exists $\tilde n\in W^{1,1}(U;\R^3)$ and a subsequence (no relabeling) such that
\[
  \begin{split}
n_h&\to \tilde n \quad \text{ in } L^1\\
\nabla n_h&\wto \nabla \tilde n \quad \text{ in } L^1\,.
\end{split}
\]
Taking into account once more \eqref{eq:6} and the fact that $n_h$ is clearly bounded in $L^\infty$, this can be upgraded to $\tilde n\in W^{1,p}(U;\R^3)$ and
  \begin{equation}\label{eq:7}
  \begin{split}
n_h&\to \tilde n \quad \text{ in } L^q \text{ for all } q \in [1,\infty)\\
\nabla n_h&\wto \nabla \tilde n \quad \text{ in } L^p\,.
\end{split}
\end{equation}

By the estimates of Lemma \ref{lma: normal estimate}, and the bounds on $\nabla n_h$, we easily obtain that also
\[
\bar n_h\to \tilde n \quad \text{ in } L^p(U)\,.
\]
Since  $\tilde n$ is the strong $L^p$ limit of the normal maps $\bar n_h$, which lie in the polar cap $\{\nn(p)\,:\,|p|\leq L\}\subset S^2$, $\tilde n$ also lies in the same polar cap, where $\nn$ is a diffeomorphism. Thus
 \[\nabla u_h =\nn^{-1}(\bar n_h) \to \nn^{-1}(\tilde n),
  \]
while at the same time $\nabla u_h\to \nabla u$ in every $L^q(U)$, $q\in[1,\infty)$. Thus $\tilde n = \nn(\nabla u)$, and since $\tilde n\in W^{1,p}(U)$, we have $u\in W^{2,p}(U)$.
\end{proof}

\begin{proof}[Proof of Theorem \ref{thm: main} (ii)] We may assume that the left hand side is finite, otherwise there is nothing to show. By part (i) of Theorem \ref{thm: main}, we have that $n_h\to \nn(\nabla u)$ for all $q\in [1,\infty)$ and $\nabla n_h\wto \nabla (\nn(\nabla u))$ in $L^p$. This is enough to prove the claimed inequality
by  standard weak lower semicontinuity results (see \cite[Theorem 3.23]{MR2361288}).
\end{proof}


We turn to the proof of the upper bound in Theorem \ref{thm: main}, which is basically contained in the following approximation lemma. In its statement, $C^2(\overline U)$ denotes the set of functions in $C^2(U)$ whose partial derivatives up to order two may be extended continuously to the closure of $U$.

\begin{lma}
\label{lma:approx}
  Let $u\in C^2(\overline U)$. For $h$ small enough the following holds true: If $\T_h$ is a regular triangulation of $U$ with $\size(\T_h)=h$, then there exists a piecewise affine continuous function $u_h$ and a normal  field $n_h$ such that $(\T_h,u_h,n_h)\in \snu$ with
    \begin{equation}\label{eq:1}
    \begin{split}
    \|u_h-u\|_{L^p}&\leq C  h^2\\
    \|\nabla u_h-\nabla u\|_{L^p}&\leq C  h\\
\|\nabla n_h- \nabla \nn(\nabla u)\|_{L^\infty}&\leq C  h\,. 
\end{split}
\end{equation}
In the above inequalities, the constant $C$ may depend on $\|\nabla^2 u\|_{L^\infty}$.
  \end{lma}

  \begin{proof}
    We define $u_h$ by letting $u_h(x)=u(x)$ for every vertex $x\in \mathscr V(\T_h)$, and by affine interpolation on every $\kappa\in\T_h$. This choice immediately yields the first two  inequalities in \eqref{eq:1}. It remains to construct $n_h\in N((u_h)_*(\T_h))$  satisfying the third estimate.

    \medskip

    For every edge $e=[x_1,x_2]$ in $\T_h$, we define $n_h$ in the edge midpoint $m=\frac{x_1+x_2}{2}$ as the vector in  $S^2$ orthogonal to $(x_2-x_1,u(x_2)-u(x_1))$ closest to $\nn(\nabla u(m))$, i.e. the unique vector  in $\{w\in S^2:w\cdot (x_2-x_1,u(x_2)-u(x_1))=0\}$ that satisfies
    \[
      \begin{split}
\left\|n_h(m)-\nn(\nabla u(m))\right\|
        =\dist\left(\{w\in S^2:w\cdot (x_2-x_1,u(x_2)-u(x_1))=0\},\nn(\nabla u(m)) \right)\,.
      \end{split}
\]
        In particular this choice guarantees that $n_h\in N((u_h)_*(\T_h))$.

\medskip

        On every $\kappa\in\T_h$, $n_h$ is now defined by affine interpolation between the edge midpoints to yield an element of $\CR(\T_h;\R^3)$.

\medskip

To show the convergence $\nabla n_h\to \nabla \nn(\nabla u)$ in $L^\infty$,  it suffices to show that at every  midpoint $m=\frac{x_1+x_2}{2}$ of an edge $e=[x_1,x_2]$, 
  \begin{equation}\label{eq:3}
|n_h(m)-\nn(\nabla u(m))|\lesssim h^2\,.
\end{equation}
Indeed, the last line of \eqref{eq:1} follows from \eqref{eq:3} easily by the fact that $\nn(\nabla u)\in C^2$ and by Taylor's theorem. 

\medskip

By definition of $\nn$, we have that 
\[
\nn(\nabla u(m))\cdot \left(x_2-x_1,(x_2-x_1)\cdot\nabla u(m)\right)=0\,.
\]
Hence the function $t\mapsto \nn(\nabla u(m))\cdot \left(x_2-x_1,(x_2-x_1)\cdot \nabla u(tx_2+(1-t)x_1)\right)$ has a zero in $t=\frac12$ and its  derivative is bounded by $Ch$, where $C$ depends on $\|\nabla^2 u\|_{L^\infty}$. Thus
\[
\left|\int_0^1 \nn(\nabla u(m))\cdot \left(x_2-x_1,(x_2-x_1)\cdot \nabla u(tx_2+(1-t)x_1)\right)\d t\right|\leq C h^3\,.
\]
This yields
\[
\left|\nn(\nabla u(m))\cdot (x_2-x_1,u(x_2)-u(x_1))\right|\leq Ch^3\,.
\]
From this estimate and the construction of $n_h(m)$ as the vector in the grand circle in $S^2$  orthogonal to $(x_2-x_1,u(x_2)-u(x_1))$ with least distance from $\nn(\nabla u(m))$, we obtain \eqref{eq:3}.  \end{proof}

\begin{proof}[Proof of Theorem \ref{thm: main} (iii)]
For $u\in C^2(\overline U)$ the upper bound statement in the theorem is an immediate 
  consequence of Lemma \ref{lma:approx}. For $u\in W^{1,\infty}\cap W^{2,p}(U)$, one uses an approximation by $C^2$-functions and a standard diagonal sequence argument.
\end{proof}

\section{Numerical example}

\label{sec:numerical}

Before presenting the numerical example of the above discretization in Section \ref{sec:numerical-experiment} below, we introduce a slight modification for numerical efficiency in Section \ref{sec:pseudo-unit-edge}.

\subsection{Pseudo-unit edge directors}
\label{sec:pseudo-unit-edge}

  In the statement of the Theorem  \ref{thm: main} 
we may replace the set of unit edge directors  by the set of \emph{pseudo-unit edge directors}, which we define below. 
\begin{defi}
A \emph{pseudo - unit edge director} of a triangular complex $\T$ is a map $n:\E(\T)\to \R^3$ such that for every $e\in\E(\T)$ with adjacent triangles $\kappa,\kappa'$, there exists $\lambda(e)\in\R$ such that
\[
  n(e) := \frac{\bar n(\kappa) +  \bar n(\kappa')}{|\bar n(\kappa) +  \bar n(\kappa')|} + \lambda(e) \big(\bar n(\kappa) - \bar n(\kappa')\big)
\]
The family of pseudo-unit edge directors is denoted $PN(\T)$ and an example is provided in Figure \ref{fig:pseudonormal}.

\begin{figure}[h]
    \centering
    \includegraphics[width=\textwidth]{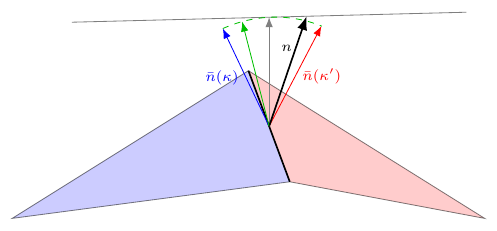}
    \caption{While the triangle normals $\bar n(\kappa)$, $\bar n(\kappa')$ are predetermined 
        by the triangle’s vertices, edge directors have one additional degree of freedom. 
        A possible unit edge director is shown in green, and a pseudo-unit edge director $n$ is shown in black.
    }
    \label{fig:pseudonormal}
\end{figure}

 \end{defi}


The motivation for defining pseudo-unit edge directors as above is that for a given $\T$, the  choice of some $n\in PN(\T)$ amounts to satisfying a set of \emph{linear} constraints. Such a choice is less computationally demanding than the choice of some $n\in N(\T)$.
To obtain the  statement of Theorem \ref{thm: main} for pseudo-edge directors, one only needs to make one notable modification in the proof, namely to replace the  comparison estimate in Lemma  \ref{lma: normal estimate} by Lemma \ref{lma:pseudo} below. We leave the other (rather obvious) modifications in the proof of Theorem \ref{thm: main}  to the reader.

\begin{lma}
\label{lma:pseudo}
Let $\T$ be a regular triangular complex.
There exists a constant $\e_0>0$ that depends only on  $C_*$ such that for
 every $n\in PN(\T)$, $e = \kappa \cap \kappa'\in \E(\T)$, if $(\diam\kappa) (|D n_\kappa| + |D n_{\kappa'}|)<\e_0$, then
  \[|n(e) - \bar n(\kappa)| \lesssim (\diam\kappa) (|D n_\kappa| + |D n_{\kappa'}|)\,.\] 

\end{lma}

\begin{proof}
We follow the proof of Lemma \ref{lma: normal estimate} up until and including \eqref{eq: cross product}. (In particular, we apply the same rescaling yielding $\diam\kappa\simeq 1$.) 
Applying \eqref{eq: cross product} also to $\kappa'$ yields $|n(e) \times \bar n(\kappa')| \lesssim |D n_{\kappa'}|$. By applying another rotation, we may change our assumptions on the explicit forms of the vectors at play to  $\frac{\bar n(\kappa)+\bar n(\kappa')}{|\bar n(\kappa)+\bar n(\kappa')|} = (0,0,1)$, $n(e) = (\alpha,0,1)$, $\bar n(\kappa) = (\sin\theta,0,\cos\theta)$, $\bar n(\kappa') = (-\sin\theta,0,\cos\theta)$, with $\alpha\in \R, \theta\in (-\pi/2,\pi/2)$. Inequality \eqref{eq: cross product} applied to  $\kappa$ and $\kappa'$ then becomes
\[
|\alpha \cos\theta- \sin \theta| \lesssim |D n_\kappa|\quad \text{ and }\quad |\alpha\cos\theta + \sin \theta| \lesssim |D n_{\kappa'}|\,.
\]

It follows that $|\alpha\cos\theta|+|\sin\theta| \lesssim |D n_\kappa|+|D n_{\kappa'}|$. If $|D n_\kappa|+|D n_{\kappa'}|$ is small enough (which we may assume by an appropriate choice of $\e_0$), then $|\sin\theta| < 1/2$ such that $|1-\cos\theta|<|\sin\theta|$.  By the triangle inequality we obtain $|n(e) -\bar  n(\kappa)|  \lesssim  |\alpha-\sin\theta|+|1-\cos\theta|\lesssim |\alpha|+|\sin\theta|\lesssim |D n_\kappa|+|D n_{\kappa'}|$.
\end{proof}

\subsection{Numerical experiment}
\label{sec:numerical-experiment}

In this section, we apply the discretized model discussed above to an example problem. 
We consider a domain $\Omega\subset\R^2$ with a $C^2$-smooth boundary constructed with 
B-splines.
The B-spline curve is defined by a set of control nodes, which in our case are chosen to 
prevent any symmetry in $\Omega$, as shown in Figure \ref{fig:domain_and_target}.
The triangular mesh $\mathscr T$ is generated automatically with the Gmsh Python API.
We consider the continuous energy 
\begin{equation}\label{eq:cont_prob}
    E_0(u) = \int_{\Omega} |D(\mathbf{n}(\nabla u))|^4 + \big(u - \sin(2x)\cos(2y)\big)^2\d x\,,
    \end{equation}
which in its discretized form becomes 
  \begin{equation}\label{eq:8}
E(u_*(\mathscr T),n) = \int_{\Omega} |Dn|^4 + \big(u - \sin(2x)\cos(2y)\big)^2\d x\,.
\end{equation}

\begin{figure}[h]
    \centering
    \begin{subfigure}{0.49\textwidth}
        \centering
        \includegraphics[width=\linewidth]{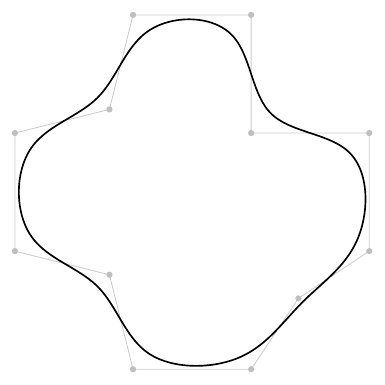}
        \label{fig:domainBsplines}
    \end{subfigure}
    \hfill
    \begin{subfigure}{0.49\textwidth}
        \centering
        \includegraphics[width=\linewidth]{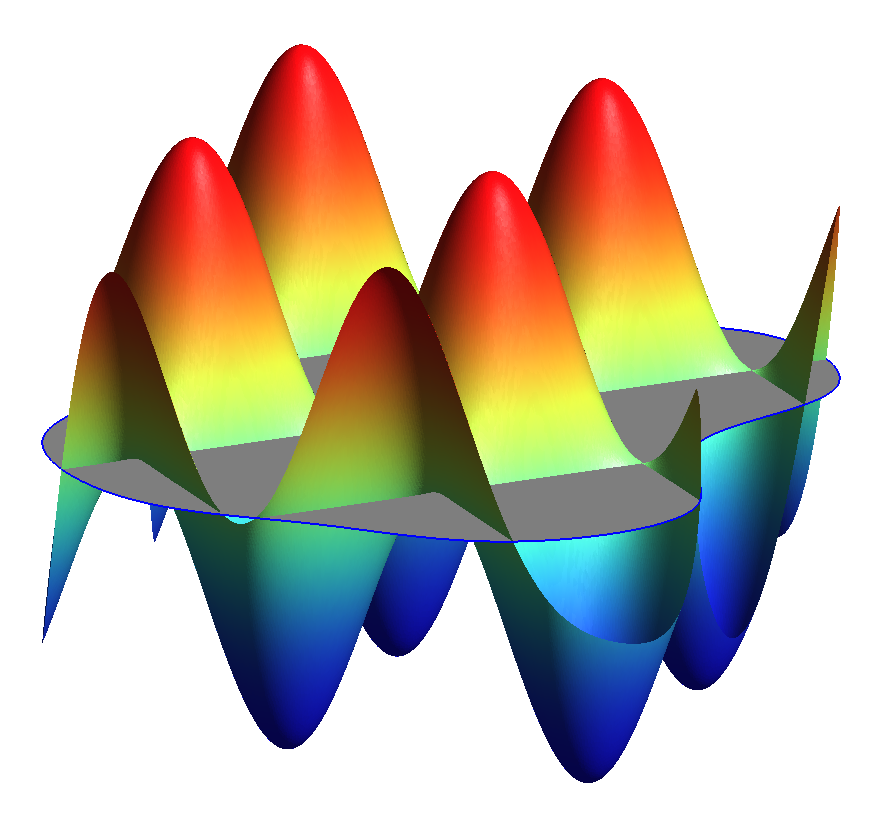}
        \label{fig:target_func}
    \end{subfigure}
    \caption{On the left, the 2D domain constructed from 3rd order B-splines with its associated nodes. On the right, the target function $\sin(2x)\cos(2y)$ evaluated on that domain.}
    \label{fig:domain_and_target}
\end{figure}

We introduce the concatenated variable $q:\mathscr V(\mathscr T)\cup \mathscr{E}(\mathscr{T})\to \R$, 
\[
q(x):=
\begin{cases}
  u(x) & \text{ if }x\in \mathscr V(\mathscr T)\\
\lambda_x & \text{ if } x\in \mathscr E(\mathscr T)
\end{cases}\,
\]
in order to condensate our notation. By slight abuse of notation, we write $E(u_*(\mathscr T),n)=E(q)$. We minimize $E$ with the L-BFGS-B routine of the SciPy Python package. This is a first-order quasi-Newton optimization method that builds an approximation of the inverse Hessian from successive gradient evaluations. As a result, it requires the computation of both the function $E$ and its gradient $\mathrm{d}E/\mathrm{d}q$, which we obtain numerically.

\medskip



\subsection{Numerical results}
The discrete surface optimal to the discrete energy $E$ is shown in Figure \ref{fig:optimal_surfaces} for decreasing mesh sizes. For the finest mesh, the mean and Gaussian curvatures 
are also presented in Figure \ref{fig:curvatures}. 

\medskip
The optimization process requires approximately one minute for the $1\,000$-triangle mesh ($h/L \approx 0.08$) and about one hour for the $15\,000$-triangle mesh ($h/L \approx 0.02$). 
Runtimes were measured on a laptop with an Intel Core i7 processor using a sequential implementation without any parallel or GPU acceleration. As illustrated in figure~\ref{fig:perf}, the computational time scales as $\mathcal{O}(h^{-3})$, as expected for a direct solver. 

\medskip
To observe the impact of the degrees of freedom offered by the edge normals, we ran the same 
simulations but fixing $\lambda(e)=0$ for all edge $e$. We will call this the \textit{reduced}
set of variables where $\tilde q = [\tilde u, 0]$. Of course, the optimum value is higher in 
this case as shown in Figure \ref{fig:perf}. There, we also verify that the discrete 
surfaces defined by the graph of $u$ and $\tilde u$ converge to different surfaces.

\medskip
In order to analyze the evolution of the $\lambda$'s, we shall first define two angles:
\begin{itemize}
    \item $\alpha_0$, the angle between the surface normal $\bar n(\kappa)$, and the mean edge normal. This is  half the \textit{dihedral angle}.
    \item $\alpha$, the angle between the pseudo-unit edge director $n(e)$ and the mean edge normal $(n_\kappa+n_{\kappa'})/|n_\kappa+n_{\kappa'}|$.
    
\end{itemize}

Figure \ref{fig:lambdas} indicates the behavior of the dihedral angle and the difference between mean edge normal and pseudo-unit edge director for different mesh sizes. As the mesh size tends to 0, the ratio $\alpha/\alpha_0$ decreases; this indicates the smoothness of the solution.  We also observe that in some cases, $\alpha > \alpha_0$. This is related to neighboring triangles whose sizes differ significantly.

\begin{figure}
    \centering
    \includegraphics[width=\linewidth]{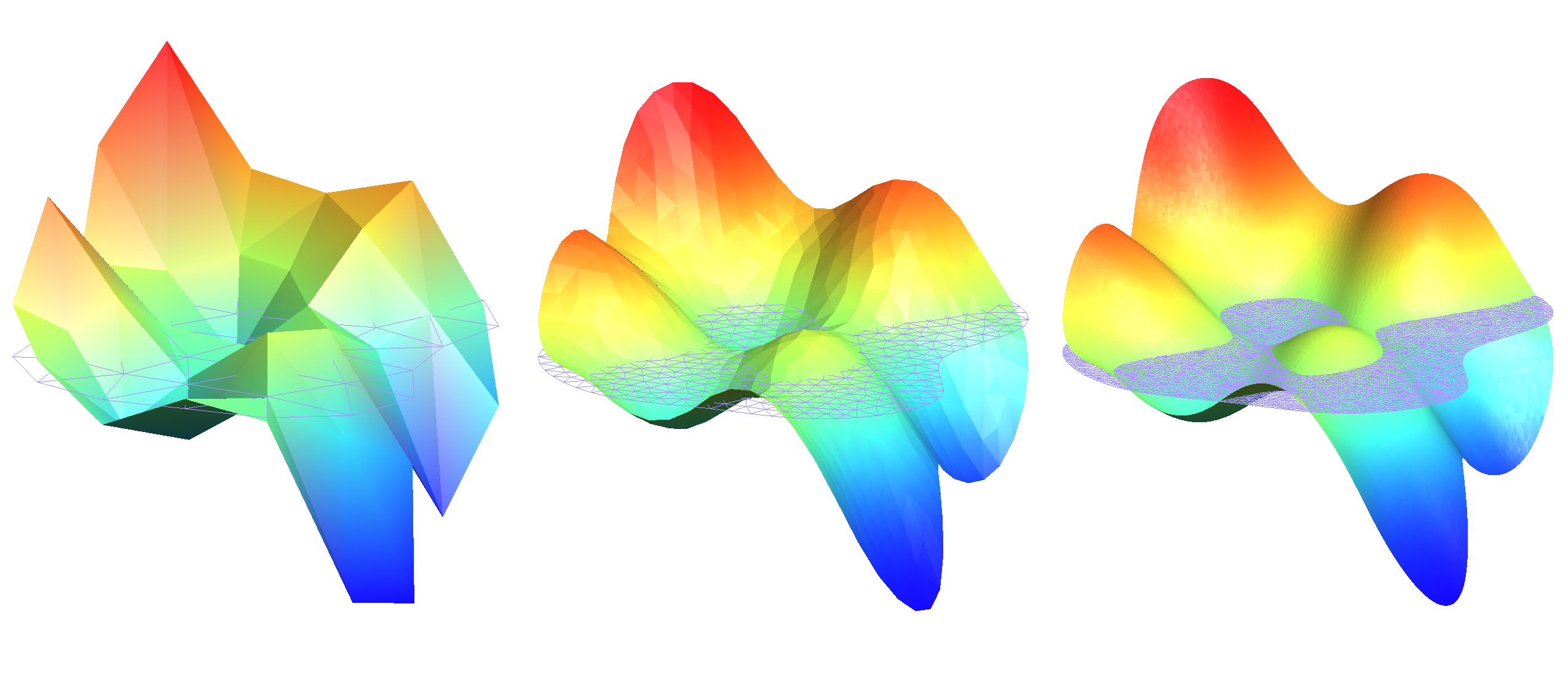}
    \caption{Surfaces minimizing the discrete energy for increasing refinement. The meshes are respectively made of $50$, $1\, 000$, and $15\, 000$ triangles.}
    \label{fig:optimal_surfaces}
\end{figure}

\begin{figure}
    \centering
    \includegraphics[width=0.75\linewidth]{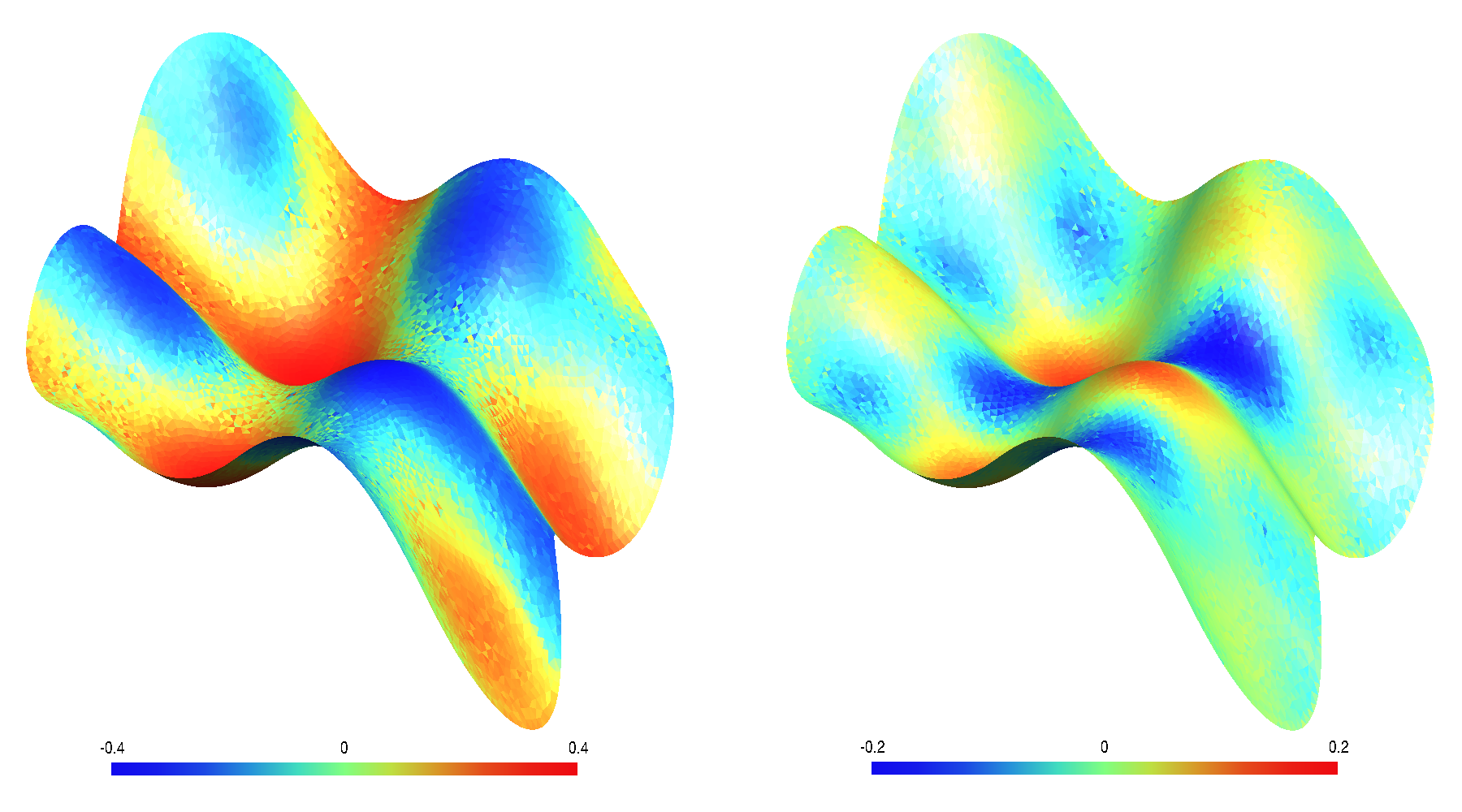}
    \caption{Mean (left) and Gaussian (right) curvatures of the optimal surface on the $15\,000$ triangles mesh.}
    \label{fig:curvatures}
\end{figure}

\begin{figure}
    \centering
    \includegraphics[width=\linewidth]{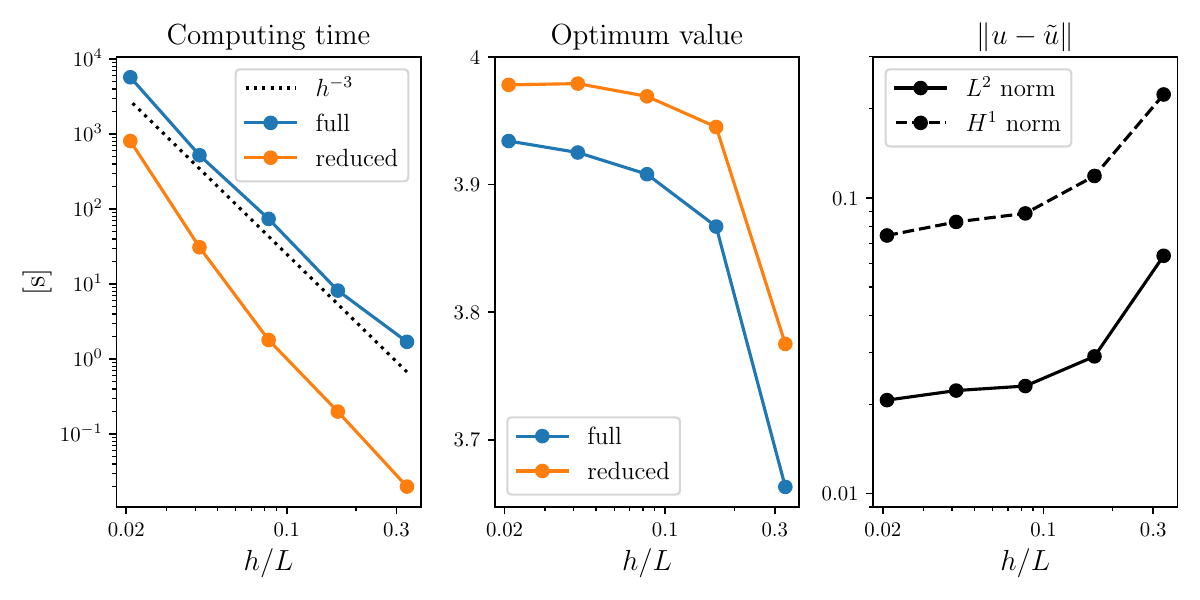}
    \caption{Computational cost and accuracy versus mesh size $h$: time scales as $\mathcal{O}(h^3)$, the objective value stabilizes with mesh refinement while the full reduced solutions $u$ and $\tilde u$ converge towards distinct optima.}
    \label{fig:perf}
\end{figure}

\begin{figure}
    \centering
    \includegraphics[width=\linewidth]{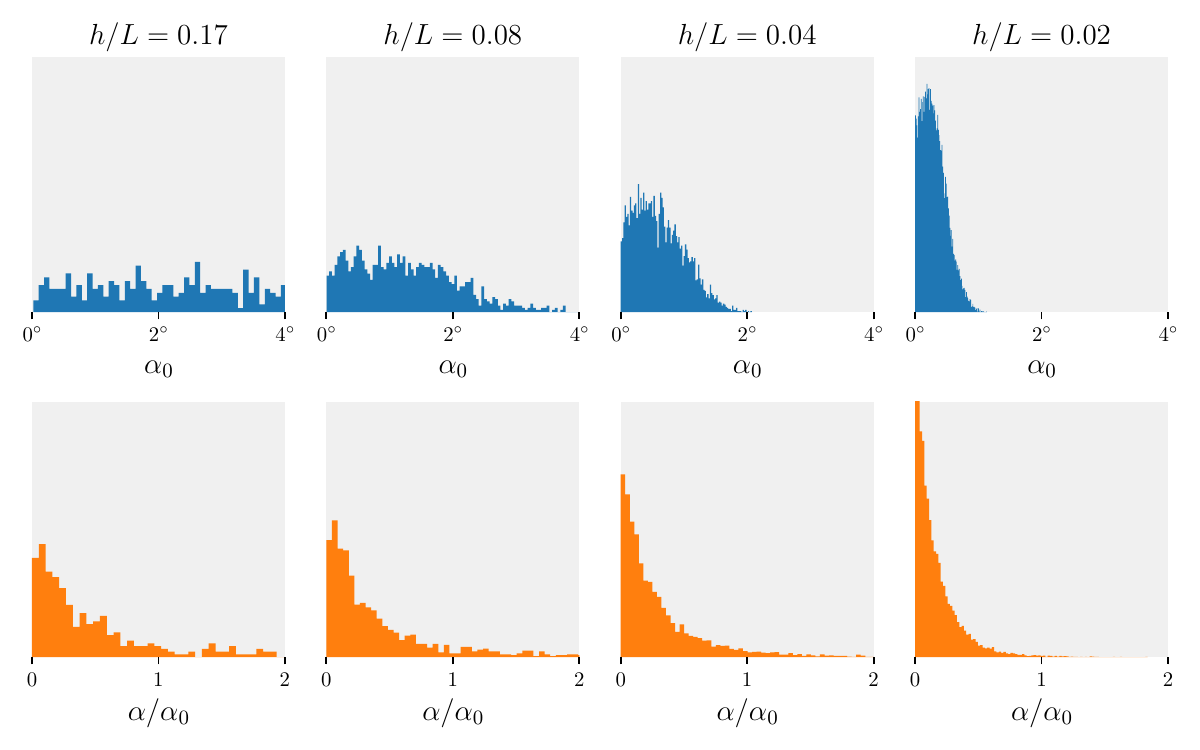}
    \caption{As the mesh is refined, the surface gets smoother. This is indicated by the decrease of the angles $\alpha_0$. The angles $\alpha$ decrease even faster as shown in the second row.}
    \label{fig:lambdas}
\end{figure}

The convergence behavior of the optimizer is displayed in Figure \ref{fig:convergence}. Each row shows the relative error between the current solution vector $q^{(k)}$ and the last solution vector $q^{(N)}$:
\begin{itemize}
    \item for the objective value $E$, the error is $|E_k-E_N|\,/\,|E_N|$,
    \item for the surface elevation $u$, the error is $\|u_k-u_N\|_{L^2} \,/\, \|u_N\|_{L^2}$,
    \item for the edge parameter $\lambda$, the error is $\|\lambda_k-\lambda_N\|_2 \,/\, \|\lambda_N\|_2$.
\end{itemize}
The left column (``full model'') shows the behavior of the algorithm treating pseudo-unit edge directors as free variables; for comparison, we show in the right column the result obtained when eliminating this freedom and setting the edge normal equal to the mean of adjacent surface normals (``reduced model''). 
We observe that the error of the field $u$ quickly drops and then slowly decreases, 
while it is the opposite for the edge parameters $\lambda$. Convergence for the full model is slower and more erratic than for the reduced one; but we recall that the reduced model does not approximate the solution of the continuous variational problem \eqref{eq:cont_prob}.

\begin{figure}
    \centering
    \includegraphics[width=0.8\linewidth]{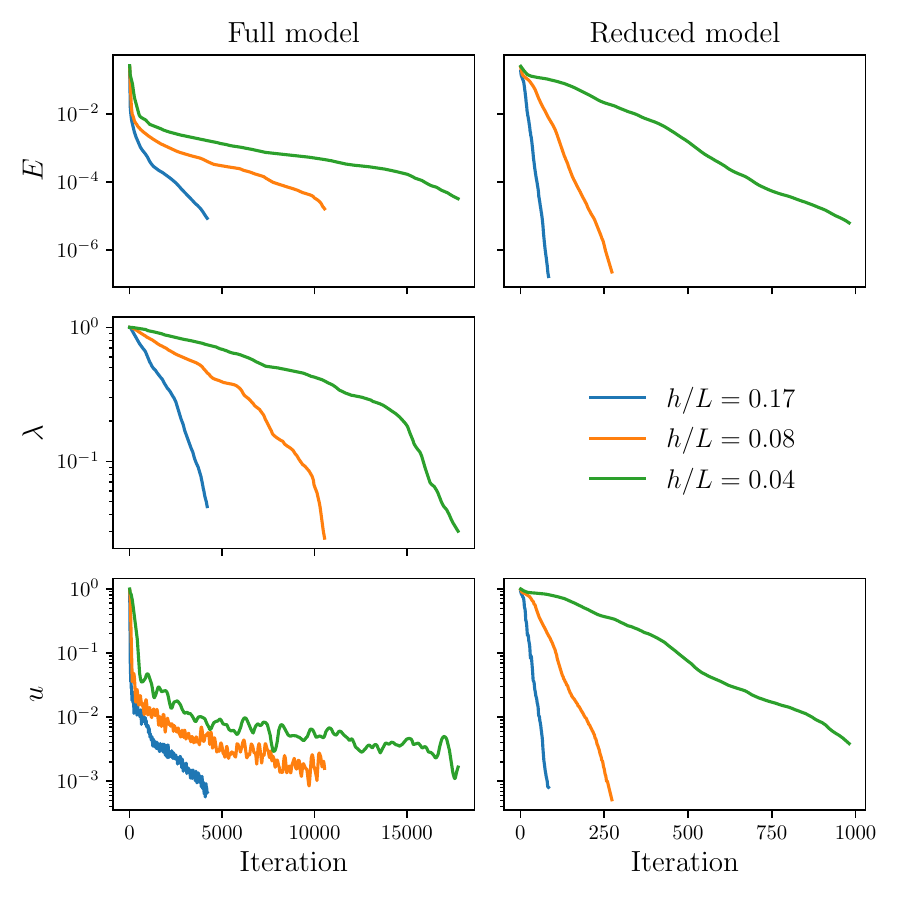}
    \caption{
        The convergence does not happen in sync for different indicators. The optimizer seems 
        to first try to converge the nodal values of $u$, before the edge parameters $\lambda$.
    }
    \label{fig:convergence}
\end{figure}






\bibliographystyle{alpha}
\bibliography{qctri}
\end{document}